\documentclass[11pt]{amsart}

\usepackage{amsmath,amssymb,amsthm,mathtools,mathrsfs}
\usepackage{enumitem}
\usepackage{hyperref}
\hypersetup{hidelinks}
\usepackage{microtype}
\usepackage{geometry}
\numberwithin{equation}{section}

\newtheorem{theorem}{Theorem}[section]
\newtheorem{proposition}[theorem]{Proposition}
\newtheorem{lemma}[theorem]{Lemma}
\newtheorem{corollary}[theorem]{Corollary}

\theoremstyle{definition}

\theoremstyle{remark}
\newtheorem{remark}[theorem]{Remark}

\newcommand{\Li}{\operatorname{Li}}
\newcommand{\BPS}{\operatorname{BPS}}
\newcommand{\wid}{\operatorname{wid}}
\newcommand{\cwid}{\operatorname{cwid}}

\newcommand{\diag}{\operatorname{diag}}

\title[Cartan-Fej\'er Gram Tomography and Flop Covariance]{Cartan-Fej\'er Gram Tomography and Flop Covariance\\
for BPS Resummed Gromov-Witten Potentials}
\author[X. Li]{Xiaobin Li}
\address{School of Mathematics, Southwest Jiaotong University,
West Zone, High-Tech District, Chengdu, Sichuan 611756, China}
\date{September 16, 2026}

\begin{document}

\begin{abstract}
We introduce a matrix-valued finite difference formalism for the genus zero
BPS resummed local Gromov-Witten potential of a threefold flop.  Factoring
the central difference as
\[
\Delta_\eta=\nabla_\eta^2,
\qquad
\nabla_\eta=T_{h\eta/2}-T_{-h\eta/2},
\]
we prove that the mixed differences
\(
\mathbf H=(\nabla_{\eta_i}\nabla_{\eta_j}F)_{i,j}
\)
admit an exact rank one signed \(q\)-Gram decomposition.  Primitive BPS
classes are therefore detected by rank one coefficient matrices, and a
matrix-valued M\"obius inversion reconstructs all multiplicities on a
primitive ray.

Under a simple threefold flop the Gram forcing is covariant away from the
flopped ray and acquires the universal logarithmic anomaly
\[
\kappa_C(C^+\otimes C^+)\log r,
\qquad
\kappa_C=\sum_{d\ge1}d^3n_{dC}.
\]
After a natural renormalization it becomes flop covariant, while one
classical logarithmic derivative recovers the cubic correction in the
crepant transformation formula.  For a smooth irreducible flopping curve
we distinguish the genus zero GV spectrum \((n_d)\) from the weighted
sequence \(w_d=d^2n_d\), which we call the GV width distribution.  Toda's
noncommutative width is the total mass \(\sum_d w_d\), whereas the flop
anomaly coefficient \(\kappa_C=\sum_d d\,w_d=\sum_d d^3n_d\) is its first
moment, equivalently the third moment of the GV spectrum.  This viewpoint
gives low degree reconstruction formulas and a cohomological width series
from contraction algebra BPS invariants.

For root supported BPS theories, simple root Gram coefficients reconstruct
the signed Cartan matrix.  Applied to the standard ADE foldings, the theory
distinguishes \(B_n\) from \(C_n\), detects the \(F_4\) double bond and
\(D_4\) triality, and yields explicit global \(q\)-Gram discriminants.
\end{abstract}

\maketitle

\section{Introduction and Main Results}

Difference equations provide a compact way to reorganize the multiple cover
expansion of local Gromov-Witten theory; the resolved conifold model was
studied by Alim \cite{Alim}.  For an effective curve class
$\beta$, set
\begin{equation}\label{eq:intro-sector}
\Phi_\beta(\lambda,t)
=
\sum_{k\ge1}
\frac{q^{k\beta}}
{k(2\sin(k\lambda/2))^2},
\qquad
q^\beta=e^{2\pi i\langle\beta,t\rangle}.
\end{equation}
The sine square denominator is the standard multiple cover factor of a
genus zero BPS state; see the original BPS interpretation of
Gopakumar-Vafa \cite{GopakumarVafaI,GopakumarVafaII}, the local
multiple-cover analysis \cite{BKL,Katz}, and the later stable-pair and
vanishing-cycle approaches to BPS/GV invariants
\cite{PandharipandeThomasStable,MaulikToda}.

For a one dimensional exceptional ray, scalar central differences cancel
this denominator and produce Fej\'er Laurent kernels.  In
\cite{LiDifferenceFlops}, the author used this mechanism to organize the
local contribution of a contractible rational curve and to reconstruct its
genus zero GV spectrum by divisor inversion.  The purpose of the present
paper is different: we retain several lattice directions simultaneously
and ask what geometric information is carried by their mixed finite
differences.

The principal new results are the following.  First, the scalar central
difference is lifted to a matrix of half step mixed differences with an
exact finite $\lambda$ rank one signed $q$-Gram decomposition; its classical
limit admits a matrix-valued divisor inversion for BPS data.  Second, and
this is the birational core of the paper, under a simple threefold flop the
Gram observable is exactly covariant off the flopped ray and acquires the
universal rank one logarithmic anomaly
\[
\Bigl(\sum_{d\ge1}d^3n_{dC}\Bigr)
(C^+\otimes C^+)\log r.
\]
Its classical logarithmic derivative is precisely the cubic correction in
the crepant transformation formula.  Third, we package the same raywise GV
data into a width distribution $w_d=d^2n_d$ and a width polynomial
$\mathsf W_C(z)=\sum_d w_dz^d$.  In this language Toda's noncommutative
width is the total mass, while the flop anomaly coefficient is the first
width moment; low degree moments reconstruct the GV spectrum and admit a
cohomological lift through contraction algebra BPS objects.  Fourth, for
root supported BPS theories the primitive Gram coefficients reconstruct
signed Cartan rows, leading to explicit non-simply-laced applications and
global $q$-Gram discriminants.  The scalar Fej\'er mechanism is taken from
\cite{LiDifferenceFlops}; the ADE deformation and folding input is taken
from \cite{BryanGholampour}, and the restricted root and flop transformation
input from \cite{NabijouWemyss}.  None of these background results is
claimed as new here.

Let $\Lambda$ be a curve class lattice, let
$\Lambda^\vee=\operatorname{Hom}(\Lambda,\mathbb Z)$, and consider a
BPS resummed contribution
\begin{equation}\label{eq:BPS-potential}
F(\lambda,t)
=
\sum_{\beta\in\mathcal S}
N_\beta\,\Phi_\beta(\lambda,t).
\end{equation}
Here $N_\beta$ denotes the genus zero BPS multiplicity.  We emphasize that
$F$ denotes the contribution obtained by resumming the multiple covers of
genus zero BPS states; no assertion about unrelated higher genus BPS
sectors is implicit.

For $\eta\in\Lambda^\vee$, put
\[
h=\frac{\lambda}{2\pi},
\qquad
\Delta_\eta=T_{h\eta}-2+T_{-h\eta}.
\]
The scalar operator $\Delta_\eta$ records
$|\langle\beta,\eta\rangle|$ but loses the sign of this pairing.  Our first
observation is that it has a canonical half step factorization
\begin{equation}\label{eq:half-factor}
\nabla_\eta=T_{h\eta/2}-T_{-h\eta/2},
\qquad
\nabla_\eta^2=\Delta_\eta.
\end{equation}
For an integral basis $\eta_1,\ldots,\eta_r$ of $\Lambda^\vee$, define
\begin{equation}\label{eq:H-intro}
\mathbf H(\lambda,q)
=
\bigl(
\nabla_{\eta_i}\nabla_{\eta_j}F
\bigr)_{1\le i,j\le r}.
\end{equation}

For $m\in\mathbb Z$, write
\[
[m]_u=\frac{u^m-u^{-m}}{u-u^{-1}}
\]
for the signed $q$-integer, understood as a Laurent polynomial for integral
$m$.  For a curve class $\beta$, define
\[
\mathbf w_\beta(u)
=
\bigl(
[\langle\beta,\eta_i\rangle]_u
\bigr)_i.
\]
\subsection*{Main results}

The first result upgrades the scalar Cartan--Fej\'er identity to an exact
matrix-valued finite-$\lambda$ decomposition.

\begin{theorem}[Quantum Gram decomposition]\label{thm:quantum-Gram}
For $u_k=e^{ik\lambda/2}$ one has
\begin{equation}\label{eq:quantum-Gram}
\mathbf H(\lambda,q)
=
-\sum_{\beta\in\mathcal S}
N_\beta
\sum_{k\ge1}
\mathbf w_\beta(u_k)\mathbf w_\beta(u_k)^T
\frac{q^{k\beta}}{k}.
\end{equation}
In particular, if $\beta$ is primitive, then
\[
-[q^\beta]\mathbf H
=
N_\beta\,\mathbf w_\beta(u)\mathbf w_\beta(u)^T,
\qquad u=e^{i\lambda/2},
\]
so every primitive coefficient matrix has rank one.  In the classical
limit,
\[
-[q^\beta]\mathbf H^{\rm cl}
=
N_\beta\,\mathbf v_\beta\mathbf v_\beta^T,
\qquad
\mathbf v_\beta=(\langle\beta,\eta_i\rangle)_i.
\]
\end{theorem}

The classical coefficient matrices also recover all BPS multiplicities on
an effective primitive ray.  Fix a primitive class $\rho$, write
$\mathbf v=\mathbf v_\rho$, and set
\[
\mathbf C_m=-[q^{m\rho}]\mathbf H^{\rm cl},\qquad m\ge1.
\]

\begin{theorem}[Tensor-valued M\"obius inversion]\label{thm:tensor-Mobius}
For every $m\ge1$,
\begin{equation}\label{eq:tensor-Mobius}
N_{m\rho}\,\mathbf v\mathbf v^T
=
\frac1{m^2}
\sum_{r\mid m}
\frac{\mu(r)}r
\mathbf C_{m/r}.
\end{equation}
Thus the classical Gram forcing reconstructs the complete BPS spectrum on
that primitive ray.
\end{theorem}

The second main result concerns birational geometry.  Threefold flops
originate in the minimal-model program and their ADE, derived and
noncommutative structures have been developed in
\cite{Reid,Kollar,KatzMorrison,BridgelandFlops,VanDenBergh}.  Let
$X\dashrightarrow X_i^+$ be a simple flop of $C_i\subset X$, with flopped
curve $C_i^+\subset X_i^+$.  Let $\mathsf M_i$ be the curve-lattice
transformation and
$\mathsf N_i=(\mathsf M_i^{-1})^\star$ its dual, as in
\cite{NabijouWemyss}, and write
\(
\mathscr H_X(\eta,\xi)=\nabla_\eta\nabla_\xi F_X
\).

\begin{theorem}[Flop covariance with rank one anomaly]\label{thm:flop-cov}
After analytic continuation and the standard change of Novikov variables,
\begin{equation}\label{eq:flop-anomaly}
\begin{aligned}
&
\mathscr H_{X_i^+}(\eta,\xi)
-
\mathscr H_X(\mathsf N_i\eta,\mathsf N_i\xi)
\\
&\qquad\equiv
\kappa_i
(C_i^+\cdot\eta)
(C_i^+\cdot\xi)
\log r_i,
\end{aligned}
\qquad
\kappa_i=\sum_{d\ge1}d^3n_{dC_i,X}.
\end{equation}
Hence the defect of covariance is supported on the flopped ray, has rank
one in the divisor variables, and is independent of $\lambda$.  After the
natural half-anomaly renormalization, the centered forcing is exactly flop
covariant.
\end{theorem}

In the classical limit,
\[
-\mathcal D_\gamma
\mathscr H^{\rm cl}(\eta,\xi)
\]
is the genus zero three point quantum potential.  Applying
$-\mathcal D_\gamma$ to \eqref{eq:flop-anomaly} gives precisely the
rank one cubic correction in the crepant transformation formula.  Hence
the finite $\lambda$ Gram anomaly is a logarithmic primitive of the cubic
wall-crossing anomaly.

The coefficient $\kappa_C$ is more naturally understood together with
the full raywise GV spectrum.  The contraction algebra was introduced by
Donovan--Wemyss \cite{DonovanWemyss}; its relation to the analytic geometry
and curve-counting invariants of flops is developed in
\cite{BrownWemyss,HuaToda}.  For a smooth irreducible flopping curve $C$,
write
\[
n_d=n_{d[C]}.
\]
We refer to the finite sequence $(n_d)_{d\ge1}$ as the genus zero GV
spectrum of $C$.  Motivated by Toda's formula
\cite{TodaWidth}, we introduce the terminology
\[
w_d(C):=d^2n_d
\]
for the \emph{GV width distribution}; this terminology is used in the
present paper for the weighted sequence $(d^2n_d)$ and is not meant to
replace the standard term GV spectrum for $(n_d)$.  Its generating
polynomial is
\begin{equation}\label{eq:width-poly-intro}
\mathsf W_C(z)=\sum_{d\ge1}w_d(C)z^d
=\sum_{d\ge1}d^2n_dz^d,
\qquad
\Theta=z\frac{d}{dz}.
\end{equation}

\begin{theorem}[GV width distribution and moment reconstruction]
\label{thm:width-main}
Assume that $C$ is a smooth irreducible flopping curve and that the genus
zero GV spectrum $(n_d)$ has finite support.  Put
$\ell=\max\{d:n_d\ne0\}$ and
\[
\kappa_C=\sum_{d\ge1}d^3n_d.
\]
Then the following hold.
\begin{enumerate}[label=\textup{(\roman*)}]
\item Toda's noncommutative and commutative widths are encoded by
\[
[z]\mathsf W_C(z)=\cwid(C)=n_1,
\qquad
\mathsf W_C(1)=\wid(C)=\dim_{\mathbb C}A_{\rm con},
\]
and the flop anomaly coefficient is
\begin{equation}\label{eq:width-moment-intro}
(\Theta\mathsf W_C)(1)=\kappa_C.
\end{equation}
Thus $\wid(C)$ is the second moment and $\kappa_C$ the third moment of the
GV spectrum; equivalently, $\wid(C)$ is the total mass and $\kappa_C$ the
first moment of the GV width distribution.

\item If $\wid(C)>0$, the normalized weights
\[
p_d=\frac{w_d(C)}{\wid(C)}
\]
define a probability distribution on the GV degrees, with
\begin{equation}\label{eq:mean-width-degree}
\mathbb E_{\rm wid}[d]
=\frac{\kappa_C}{\wid(C)}.
\end{equation}
In particular,
\[
\wid(C)\le \kappa_C\le \ell\,\wid(C).
\]

\item If $n_d=0$ for $d\ge3$, then the third GV moment carries no new
numerical information beyond the two widths:
\begin{equation}\label{eq:length2-main}
\kappa_C=2\wid(C)-\cwid(C).
\end{equation}
If $n_d=0$ for $d\ge4$, then the triple
$(\cwid(C),\wid(C),\kappa_C)$ reconstructs the complete spectrum by
\begin{equation}\label{eq:length3-main}
 n_1=\cwid(C),\qquad
 n_3=\frac{\kappa_C-2\wid(C)+\cwid(C)}9,
\qquad
 n_2=\frac{3\wid(C)-\kappa_C-2\cwid(C)}4.
\end{equation}
Consequently degree three is the first level at which the Gram anomaly can
supply an independent width moment.

\item Let $\BPS_{A_{\rm con},d}$ be Davison's cohomological BPS objects for
the contraction algebra, with
$\dim\BPS_{A_{\rm con},d}=n_d$ \cite{Davison}.  Then
\begin{equation}\label{eq:coh-width-intro}
\mathbb W_C(z)
=\sum_{d\ge1}d^2[\BPS_{A_{\rm con},d}]z^d
\end{equation}
defines a cohomological width series whose dimension realization is
$\mathsf W_C(z)$.  Its first width moment
\[
\mathbb K_C
=\left.\Theta\mathbb W_C(z)\right|_{z=1}
=\sum_{d\ge1}d^3[\BPS_{A_{\rm con},d}]
\]
satisfies
\begin{equation}\label{eq:coh-kappa-intro}
\dim\mathbb K_C=\kappa_C.
\end{equation}
\end{enumerate}
\end{theorem}

Theorem~\ref{thm:width-main} isolates the enumerative content of the
rank one wall-crossing term in Theorem~\ref{thm:flop-cov}.  The chain of
ideas is
\[
\boxed{
\begin{gathered}
\text{width polynomial}\ \longrightarrow\ \text{third GV moment}\\
\longrightarrow\ \text{rank one flop anomaly}\ \longrightarrow\
\text{cohomological width}
\end{gathered}}
\]
The point is not merely terminological.  The ordinary width
$\wid(C)=\sum d^2n_d$ records the total weighted mass of the GV spectrum,
whereas $\kappa_C=\sum d^3n_d$ is more sensitive to higher degree BPS
sectors.  The flop anomaly therefore detects information that is invisible
to the total width once degree three sectors are allowed.

Our fourth theme is root theoretic reconstruction.  Suppose the BPS support
is a reduced root system with simple roots $\alpha_1,\ldots,\alpha_r$, take
the simple coroots as difference directions, and let
$A=(a_{ij})$ be the Cartan matrix.  Put
\[
\mathbf c_j(u)=([a_{j1}]_u,\ldots,[a_{jr}]_u)^T,
\qquad
A_u=([a_{ji}]_u)_{j,i},
\]
\[
D_N=\operatorname{diag}(N_{\alpha_1},\ldots,N_{\alpha_r}),
\qquad
\mathcal G(u)=-\sum_{j=1}^r[q^{\alpha_j}]\mathbf H.
\]

\begin{theorem}[Cartan reconstruction and global $q$-Gram factorization]
\label{thm:Cartan-main}\label{thm:qCartan-row}\label{thm:signed-Dynkin}\label{thm:global-Gram}
For every simple root $\alpha_j$,
\begin{equation}\label{eq:qCartan-row}
-[q^{\alpha_j}]\mathbf H
=
N_{\alpha_j}\,\mathbf c_j(u)\mathbf c_j(u)^T.
\end{equation}
If $N_{\alpha_j}\ne0$ for all simple roots, these Laurent polynomial
profiles determine the complete signed Cartan matrix through
\begin{equation}\label{eq:signed-ratio}
\frac{
\bigl(-[q^{\alpha_j}]\mathbf H\bigr)_{ij}
}{
\bigl(-[q^{\alpha_j}]\mathbf H\bigr)_{jj}
}
=
\frac{[a_{ji}]_u}{[2]_u},
\end{equation}
and, at $u=1$,
\begin{equation}\label{eq:signed-classical}
a_{ji}
=
2
\frac{
\bigl(-[q^{\alpha_j}]\mathbf H^{\rm cl}\bigr)_{ij}
}{
\bigl(-[q^{\alpha_j}]\mathbf H^{\rm cl}\bigr)_{jj}
}.
\end{equation}
Moreover,
\begin{equation}\label{eq:global-Gram}
\mathcal G(u)=A_u^TD_NA_u,
\end{equation}
so
\[
\det\mathcal G(u)
=
\left(\prod_jN_{\alpha_j}\right)\det(A_u)^2.
\]
\end{theorem}

As applications, we use standard ADE deformation and folding results of
Bryan-Gholampour \cite{BryanGholampour}, together with the classical
ADE/invariant-theoretic background \cite{KatzMorrison,Slodowy} and the
restricted-root description of threefold flops in \cite{NabijouWemyss}.  These are inputs,
not claims of novelty.  Restriction to invariant K\"ahler loci gives the
standard foldings
\[
A_{2n-1}\to C_n,\qquad
D_{n+1}\to B_n,\qquad
E_6\to F_4,\qquad
D_4\to G_2.
\]
The resulting Gram forcing detects root lengths and bond orientations,
distinguishes $B_n$ from $C_n$, and records $D_4$ triality by the signed
$q$-integer $[-3]_u$.  The $B_n/C_n$ family admits the following closed
form discriminant, which will be proved in Section~4.

\begin{theorem}[$B/C$ determinant]\label{thm:BC-det}
For the standard signed $q$-Cartan matrices of types $B_n$ and $C_n$, set
\(
\delta_n(u)=\det A_u(B_n)=\det A_u(C_n)
\).
Then, for $n\ge2$,
\begin{equation}\label{eq:BC-det}
\delta_n(u)
=
[2]_u\bigl([n]_u-[n-1]_u\bigr)
=
u^{-n}\frac{(u^2+1)(u^{2n-1}+1)}{u+1}.
\end{equation}
\end{theorem}

The paper is organized in four sections.  The present section states the
main results.  Section~2 develops the
Cartan--Fej\'er difference formalism and centered Gram tomography, beginning
with the geometric setup and ending with tensor-valued divisor inversion.
Section~3 proves flop covariance, identifies the logarithmic rank one
wall-crossing anomaly, and proves Theorem~\ref{thm:width-main}, including
its numerical moment reconstruction and cohomological width lift.
Section~4 reconstructs the signed Cartan data from primitive Gram
coefficients, applies the construction to the four standard
non-simply-laced ADE foldings and their global $q$-Gram discriminants, and
concludes with further directions.

\section{Cartan--Fej\'er differences and centered Gram tomography}

\subsection{Geometric and formal setup}

\subsubsection{Formal difference setup}

Let $\Lambda$ be a free abelian group of rank $r$, with dual lattice
\[
\Lambda^\vee=\operatorname{Hom}(\Lambda,\mathbb Z).
\]
Fix a pointed effective monoid $\Lambda_+\subset\Lambda$.  All Novikov
series are understood in the completion of the group algebra with respect
to $\Lambda_+$.

The K\"ahler variable $t$ is taken on its universal cover.  For
\[
q^\beta=e^{2\pi i\langle\beta,t\rangle}
\]
and $\eta\in\Lambda^\vee$, the half step translation is defined
coefficientwise by
\begin{equation}\label{eq:half-shift-def}
T_{h\eta/2}q^\beta
=
e^{i\lambda\langle\beta,\eta\rangle/2}q^\beta.
\end{equation}
Thus $\eta/2$ need not belong to the integral dual lattice.

We use the BPS sectors \eqref{eq:intro-sector} and the potential
\eqref{eq:BPS-potential} throughout.

\subsubsection{Threefold flops and restricted roots}

For a smooth threefold flopping contraction, the classical birational
framework goes back to \cite{Reid,Kollar,KatzMorrison}; derived and
noncommutative descriptions were developed in
\cite{BridgelandFlops,VanDenBergh}.  The nonzero genus-zero GV classes are
governed by ADE restricted roots.  Nabijou-Wemyss
\cite{NabijouWemyss} characterize the nonzero GV support in this way and
describe its transformation under flop; related GV wall-crossing is also
developed in \cite{TodaWallCrossing}.  They also construct the relevant
curve and divisor lattice transformations and prove the crepant
transformation formula used in Section~3.

For the standard ADE surface resolutions, Bryan-Gholampour
\cite{BryanGholampour} use generic deformation to reduce the curve theory
to isolated $(-1,-1)$ curves indexed by positive roots.  Their root-system
formulas and standard non-simply-laced foldings are used as input in
Section~4, in the classical folding framework of \cite{Slodowy}.

\subsubsection{Admissible foldings}

For the admissible foldings
\[
A_{2n-1}/\mathbb Z_2\to C_n,\qquad
D_{n+1}/\mathbb Z_2\to B_n,
\]
\[
E_6/\mathbb Z_2\to F_4,\qquad
D_4/\mathbb Z_3\to G_2,
\]
let $\mathcal O_\alpha$ be an ADE root orbit and set
\[
\bar\alpha
=
\frac1{|\mathcal O_\alpha|}
\sum_{\alpha'\in\mathcal O_\alpha}\alpha'.
\]
For these admissible foldings, roots in one orbit are pairwise orthogonal;
compare the ADE folding and simultaneous-resolution viewpoints in
\cite{Slodowy,KatzMorrison,BryanGholampour}.
If ADE roots have squared length two, then
\begin{equation}\label{eq:orbit-length}
(\bar\alpha,\bar\alpha)
=
\frac2{|\mathcal O_\alpha|},
\qquad
\bar\alpha^\vee
=
\sum_{\alpha'\in\mathcal O_\alpha}\alpha'.
\end{equation}
Thus orbit averages give folded roots and orbit sums give folded coroots.

In the standard unit ADE root potential, restriction to the invariant
K\"ahler locus groups sectors orbitwise.  The resulting folded coefficient
is
\begin{equation}\label{eq:folded-multiplicity}
d_\beta
=
|\mathcal O_\beta|
=
\frac2{(\beta,\beta)}.
\end{equation}

\subsection{Scalar Cartan--Fej\'er differences}

For $m\ge1$, define the Fej\'er Laurent kernel
\begin{equation}\label{eq:Fejer}
K_m(z)
=
\sum_{j=-(m-1)}^{m-1}
(m-|j|)z^j,
\qquad
K_0=0.
\end{equation}
Then
\[
K_m(z)
=
z^{-(m-1)}(1+z+\cdots+z^{m-1})^2
\]
and
\begin{equation}\label{eq:Fejer-circle}
K_m(e^{ix})
=
\left(
\frac{\sin(mx/2)}{\sin(x/2)}
\right)^2.
\end{equation}

For $\eta\in\Lambda^\vee$, define
\[
\Delta_\eta=T_{h\eta}-2+T_{-h\eta}.
\]

\begin{lemma}[Scalar Cartan--Fej\'er identity]\label{lem:scalar-CF}
Let $a=\langle\beta,\eta\rangle$.  Then
\[
\Delta_\eta\Phi_\beta
=
-
\sum_{k\ge1}
K_{|a|}(e^{ik\lambda})
\frac{q^{k\beta}}k.
\]
\end{lemma}

\begin{proof}
On the mode $q^{k\beta}$,
\[
\Delta_\eta
=
e^{ika\lambda}-2+e^{-ika\lambda}
=
-4\sin^2(ka\lambda/2).
\]
Division by the multiple cover denominator and
\eqref{eq:Fejer-circle} give the result.
\end{proof}

This is the multivariable form of the scalar Fej\'er mechanism developed
in \cite{LiDifferenceFlops}.  We recall it only to motivate the mixed
construction below.

\subsection{Centered Gram tomography}

\subsubsection{\texorpdfstring{Signed $q$-integers and mixed differences}{Signed q-integers and mixed differences}}

For $\eta\in\Lambda^\vee$, define
\[
\nabla_\eta=T_{h\eta/2}-T_{-h\eta/2}.
\]
Then $\nabla_\eta^2=\Delta_\eta$.

For $m\in\mathbb Z$, set
\[
[m]_u=\frac{u^m-u^{-m}}{u-u^{-1}}.
\]
For integral $m$ this is a Laurent polynomial, regular at $u=\pm1$.
Moreover
\[
[-m]_u=-[m]_u,
\qquad
K_{|m|}(u^2)=[m]_u^2.
\]

\begin{lemma}[Centered mixed difference identity]\label{lem:mixed-CF}
Let
\[
a=\langle\beta,\eta\rangle,
\qquad
b=\langle\beta,\xi\rangle,
\qquad
u_k=e^{ik\lambda/2}.
\]
Then
\begin{equation}\label{eq:mixed-identity}
\nabla_\eta\nabla_\xi\Phi_\beta
=
-
\sum_{k\ge1}
[a]_{u_k}[b]_{u_k}
\frac{q^{k\beta}}k.
\end{equation}
\end{lemma}

\begin{proof}
On $q^{k\beta}$,
\[
\nabla_\eta
=
2i\sin(ka\lambda/2),
\qquad
\nabla_\xi
=
2i\sin(kb\lambda/2).
\]
Their product is
\[
-4\sin(ka\lambda/2)\sin(kb\lambda/2).
\]
After division by $(2\sin(k\lambda/2))^2$, the quotient is
$-[a]_{u_k}[b]_{u_k}$.
\end{proof}

\subsubsection{Quantum Gram decomposition}

Fix an integral basis $\eta_1,\ldots,\eta_r$ of $\Lambda^\vee$ and define
\[
\mathbf H(\lambda,q)
=
(H_{ij})_{i,j},
\qquad
H_{ij}
=
\nabla_{\eta_i}\nabla_{\eta_j}F.
\]
For $\beta\in\Lambda$, define
\[
\mathbf w_\beta(u)
=
\begin{pmatrix}
[\langle\beta,\eta_1\rangle]_u\\
\vdots\\
[\langle\beta,\eta_r\rangle]_u
\end{pmatrix}.
\]

\begin{proof}[Proof of Theorem~\ref{thm:quantum-Gram}]
Apply Lemma~\ref{lem:mixed-CF} to every matrix entry.
\end{proof}

\begin{corollary}[Primitive rank one identity]\label{cor:primitive-rank-one}
If $\beta$ is primitive, then
\begin{equation}\label{eq:primitive-qgram}
-[q^\beta]\mathbf H
=
N_\beta
\mathbf w_\beta(u)\mathbf w_\beta(u)^T,
\qquad
u=e^{i\lambda/2}.
\end{equation}
Hence every $2\times2$ minor vanishes.
\end{corollary}

\begin{remark}[Basis dependence]
At finite $\lambda$, $\mathbf H$ is a matrix-valued $q$-Gram observable
attached to the chosen integral difference basis.  Since
$[m+n]_u\ne[m]_u+[n]_u$ in general, it is not an ordinary bilinear tensor
under arbitrary integral basis changes.  The classical limit below is
tensorial.
\end{remark}

\subsubsection{Classical charge tensors}

Let $\mathbf H^{\rm cl}=\mathbf H|_{\lambda=0}$ and set
\[
\mathbf v_\beta
=
(\langle\beta,\eta_1\rangle,\ldots,
\langle\beta,\eta_r\rangle)^T.
\]

\begin{corollary}[Classical Gram decomposition]\label{cor:classical-Gram}
\begin{equation}\label{eq:classical-Gram}
\mathbf H^{\rm cl}(q)
=
-
\sum_{\beta\in\mathcal S}
N_\beta
\mathbf v_\beta\mathbf v_\beta^T
\Li_1(q^\beta).
\end{equation}
For primitive $\beta$,
\[
-[q^\beta]\mathbf H^{\rm cl}
=
N_\beta\mathbf v_\beta\mathbf v_\beta^T.
\]
\end{corollary}

If $\beta$ is primitive in $\Lambda$, then $\mathbf v_\beta$ is primitive
in the coordinate lattice.  Therefore
\begin{equation}\label{eq:gcd-multiplicity}
N_\beta
=
\gcd_{i,j}
\left|
[q^\beta]H_{ij}^{\rm cl}
\right|.
\end{equation}
After division by $N_\beta$, the rank one matrix determines the charge
vector up to sign.  A chosen effective chamber fixes the sign.

\subsubsection{Tensor-valued divisor inversion}

Fix a primitive effective class $\rho$ and set $\mathbf v=\mathbf v_\rho$.
For $m\ge1$ define
\[
\mathbf C_m
=
-[q^{m\rho}]\mathbf H^{\rm cl}.
\]
The sector $d\rho$ contributes to $q^{m\rho}$ exactly when $d\mid m$,
and $\mathbf v_{d\rho}=d\mathbf v$.  Hence
\begin{equation}\label{eq:tensor-convolution}
\mathbf C_m
=
\frac1m
\sum_{d\mid m}
d^3N_{d\rho}\,
\mathbf v\mathbf v^T.
\end{equation}

\begin{proof}[Proof of Theorem~\ref{thm:tensor-Mobius}]
Multiply \eqref{eq:tensor-convolution} by $m$ and apply ordinary
M\"obius inversion to the divisor convolution.
\end{proof}

\section{Flop covariance and BPS width moments}

\subsection{Flop transformations}

Let
\[
X\dashrightarrow X_i^+
\]
be a simple flop of the irreducible curve $C_i\subset X$, and let
$C_i^+\subset X_i^+$ be the flopped curve.  We use the curve-lattice
transformation
\[
\mathsf M_i:A_1(X_i^+)\to A_1(X)
\]
and its dual
\[
\mathsf N_i=(\mathsf M_i^{-1})^\star:
H^2(X_i^+)\to H^2(X)
\]
from \cite{NabijouWemyss}; compare also the broader GV wall-crossing
framework in \cite{TodaWallCrossing}.  They satisfy
\[
\mathsf M_i(C_i^+)=-C_i
\]
and
\begin{equation}\label{eq:MN-duality}
\mathsf N_i\eta\cdot\beta
=
\eta\cdot\mathsf M_i^{-1}\beta.
\end{equation}
The Novikov variables are related by
\begin{equation}\label{eq:Novikov-flop}
q^\beta
=
r^{\mathsf M_i^{-1}\beta},
\qquad
q_i=r_i^{-1}.
\end{equation}
The GV invariants are reindexed off the flopped ray, while on the ray
\[
n_{dC_i^+,X_i^+}
=
n_{dC_i,X}.
\]

For divisor directions $\eta,\xi$, write
\[
\mathscr H_X(\eta,\xi)
=
\nabla_\eta\nabla_\xi F_X.
\]
All analytic continuations below are taken on a simply connected covering
of the relevant complement of the pole arrangement.  We write
\[
A\equiv B
\]
if $A-B$ is independent of all Novikov variables.

\subsection{Covariance away from the flopped ray}

\begin{proposition}[Off-ray covariance]\label{prop:offray}
After the variable change \eqref{eq:Novikov-flop},
\[
\mathscr H_{X_i^+}^{\circ}(\eta,\xi)
=
\mathscr H_X^{\circ}(\mathsf N_i\eta,\mathsf N_i\xi),
\]
where $\circ$ denotes the contribution of classes not lying on the
flopped ray.
\end{proposition}

\begin{proof}
For an off-ray class $\gamma$ on $X_i^+$, the corresponding class on $X$
is $\mathsf M_i\gamma$.  The GV multiplicity is unchanged under this
reindexing, \eqref{eq:MN-duality} identifies the two lattice pairings, and
$q^{\mathsf M_i\gamma}=r^\gamma$.  Thus the summands in
\eqref{eq:quantum-Gram} agree term by term.
\end{proof}

\subsection{Inversion of the centered profile}

For integers $a,b$, define
\[
\mathscr L_{a,b}(Q;\lambda)
=
\sum_{k\ge1}
[a]_{u_k}[b]_{u_k}\frac{Q^k}{k}.
\]

\begin{lemma}[Centered profile inversion]\label{lem:centered-inversion}
After analytic continuation from $Q=0$ to $Q=\infty$,
\begin{equation}\label{eq:centered-inversion}
\mathscr L_{a,b}(Q^{-1};\lambda)
\equiv
\mathscr L_{a,b}(Q;\lambda)
+
ab\log Q.
\end{equation}
\end{lemma}

\begin{proof}
The Laurent polynomial $P_{a,b}(u)=[a]_u[b]_u$ is symmetric:
$P_{a,b}(u^{-1})=P_{a,b}(u)$.  Write
\[
P_{a,b}(u)=\sum_{m=-M}^M c_mu^m,
\qquad
c_m=c_{-m}.
\]
Then
\[
\mathscr L_{a,b}(Q;\lambda)
=
\sum_m c_m\Li_1(Qe^{im\lambda/2}).
\]
Using
\[
\Li_1(z^{-1})
\equiv
\Li_1(z)+\log z,
\]
symmetry cancels the linear phase term, and
\[
\sum_mc_m=P_{a,b}(1)=ab.
\]
This proves \eqref{eq:centered-inversion}.
\end{proof}

\subsection{The rank one flop anomaly}

Set
\begin{equation}\label{eq:kappa}
\kappa_i
=
\sum_{d\ge1}d^3n_{dC_i,X}.
\end{equation}

\begin{proof}[Proof of Theorem~\ref{thm:flop-cov}]
By Proposition~\ref{prop:offray}, only the flopped ray contributes to the
difference.  Put
\[
a=C_i^+\cdot\eta,\qquad
b=C_i^+\cdot\xi.
\]
Since $\mathsf M_i(C_i^+)=-C_i$,
\[
C_i\cdot\mathsf N_i\eta=-a,\qquad
C_i\cdot\mathsf N_i\xi=-b.
\]
The two signs cancel in the product of signed $q$-integers.

For the degree $d$ ray sector, the new chamber contains
$-\mathscr L_{da,db}(r_i^d;\lambda)$, while the analytically continued old
chamber contains
$-\mathscr L_{da,db}(r_i^{-d};\lambda)$.  By
Lemma~\ref{lem:centered-inversion},
\[
\mathscr L_{da,db}(r_i^{-d};\lambda)
\equiv
\mathscr L_{da,db}(r_i^d;\lambda)
+
d^3ab\log r_i.
\]
Multiplying by $n_{dC_i}$ and summing over $d$ gives
\eqref{eq:flop-anomaly}.
\end{proof}

Thus the defect of covariance is supported on one ray, has rank one in the
divisor variables, is independent of $\lambda$, and is controlled by the
single cubic GV moment $\kappa_i$.

\subsection{A length two Laufer wall}

We illustrate Theorem~\ref{thm:flop-cov} on the standard length two Laufer
geometry \cite{Laufer} and its spectrum \cite{BrownWemyss,Collinucci}
\[
(n_1,n_2)=(5,1).
\]
Then
\[
\kappa_C
=
1^3n_1+2^3n_2
=
5+8
=
13.
\]
Let
\[
a=C^+\cdot\eta,
\qquad
b=C^+\cdot\xi.
\]
The flopped ray contribution in the new chamber is
\[
-5\,\mathscr L_{a,b}(r;\lambda)
-\mathscr L_{2a,2b}(r^2;\lambda),
\]
whereas the analytically continued old chamber contribution is
\[
-5\,\mathscr L_{a,b}(r^{-1};\lambda)
-\mathscr L_{2a,2b}(r^{-2};\lambda).
\]
Lemma~\ref{lem:centered-inversion} gives
\[
\mathscr L_{a,b}(r^{-1};\lambda)
\equiv
\mathscr L_{a,b}(r;\lambda)+ab\log r
\]
and
\[
\mathscr L_{2a,2b}(r^{-2};\lambda)
\equiv
\mathscr L_{2a,2b}(r^2;\lambda)+8ab\log r.
\]
Therefore
\[
\boxed{
\mathscr H_{X^+}^{\rm ray}(\eta,\xi)
-
\mathscr H_X^{\rm ray}(\mathsf N\eta,\mathsf N\xi)
\equiv
13ab\log r.
}
\]
This example exhibits explicitly how the degree two BPS sector contributes
eight units to the logarithmic anomaly, whereas it contributes only four
units to the ordinary noncommutative width.  Thus the wall-crossing
coefficient detects the third rather than the second GV moment.

\subsection{Renormalized covariance}

Define
\[
\widehat{\mathscr H}_{X,i}(\eta,\xi)
=
\mathscr H_X(\eta,\xi)
-
\frac{\kappa_i}{2}
(C_i\cdot\eta)(C_i\cdot\xi)\log q_i
\]
and
\[
\widehat{\mathscr H}_{X_i^+,i}(\eta,\xi)
=
\mathscr H_{X_i^+}(\eta,\xi)
-
\frac{\kappa_i}{2}
(C_i^+\cdot\eta)(C_i^+\cdot\xi)\log r_i.
\]

\begin{corollary}[Renormalized flop covariance]\label{cor:renorm}
\[
\widehat{\mathscr H}_{X_i^+,i}(\eta,\xi)
\equiv
\widehat{\mathscr H}_{X,i}
(\mathsf N_i\eta,\mathsf N_i\xi).
\]
\end{corollary}

\begin{proof}
Use $q_i=r_i^{-1}$ and the equality of the two rank one tensors under the
dual lattice identification.  The two half anomalies cancel
\eqref{eq:flop-anomaly}.
\end{proof}

\subsection{Classical derivative and the crepant transformation correction}

At $\lambda=0$,
\[
\mathscr H_X^{\rm cl}(\eta,\xi)
=
-
\sum_\beta
n_{\beta,X}
(\beta\cdot\eta)(\beta\cdot\xi)\Li_1(q^\beta).
\]
For $\gamma\in H^2(X)$, define the logarithmic Novikov derivative by
\[
\mathcal D_\gamma q^\beta
=
(\beta\cdot\gamma)q^\beta.
\]
Then
\[
-\mathcal D_\gamma
\mathscr H_X^{\rm cl}(\eta,\xi)
\]
is the nonconstant genus zero three point quantum potential.

Applying $-\mathcal D_\gamma$ to Theorem~\ref{thm:flop-cov} at
$\lambda=0$ gives
\[
\begin{aligned}
&
\Phi_{X_i^+}(\eta,\xi,\gamma)
-
\Phi_X(
\mathsf N_i\eta,
\mathsf N_i\xi,
\mathsf N_i\gamma)
\\
&\qquad=
-\kappa_i
(C_i^+\cdot\eta)
(C_i^+\cdot\xi)
(C_i^+\cdot\gamma),
\end{aligned}
\]
which is the rank one correction in the crepant transformation formula.
Thus the logarithmic Gram anomaly is a primitive of the cubic correction.

\subsection{Iterated flops and the chamber cocycle}

Consider a flop sequence
\[
X_0\dashrightarrow X_1\dashrightarrow\cdots\dashrightarrow X_m.
\]
Let $\mathsf N_s:H^2(X_s)\to H^2(X_{s-1})$ be the divisor transformation
at the $s$-th wall and set
\[
\mathsf N_{[1,m]}=\mathsf N_1\cdots\mathsf N_m.
\]
For $\eta\in H^2(X_m)$, put
\[
\eta^{(s)}=\mathsf N_{s+1}\cdots\mathsf N_m\eta.
\]

\begin{corollary}[Iterated covariance]\label{cor:iterated}
\[
\begin{aligned}
&
\mathscr H_{X_m}(\eta,\xi)
-
\mathscr H_{X_0}
(\mathsf N_{[1,m]}\eta,\mathsf N_{[1,m]}\xi)
\\
&\qquad\equiv
\sum_{s=1}^m
\kappa_s
(C_s^+\cdot\eta^{(s)})
(C_s^+\cdot\xi^{(s)})
\log r_s.
\end{aligned}
\]
\end{corollary}

\begin{proof}
Apply Theorem~\ref{thm:flop-cov} at each wall and telescope.
\end{proof}

Consequently the rank one logarithmic anomalies form an exact additive
wall-crossing cocycle on the marked chamber groupoid.  We do not claim
that this defines a nontrivial cohomology class.

\subsection{Affine wall-crossing transport}

It is useful to package the single wall formula as an affine transition
law.  For an oriented flop edge
\[
e:X_-\dashrightarrow X_+
\]
with divisor transformation $\mathsf N_e$, define
\[
\mathfrak a_e(\eta,\xi)
=
\kappa_e(C_e^+\cdot\eta)(C_e^+\cdot\xi)\log r_e
\]
and set
\begin{equation}\label{eq:affine-transport}
\mathscr T_e(B)(\eta,\xi)
=
B(\mathsf N_e\eta,\mathsf N_e\xi)
+
\mathfrak a_e(\eta,\xi).
\end{equation}
Then Theorem~\ref{thm:flop-cov} is precisely
\[
\mathscr H_{X_+}\equiv\mathscr T_e(\mathscr H_{X_-}).
\]
For a flop path $\mathcal P=e_m\cdots e_1$, define
$\mathscr T_{\mathcal P}=\mathscr T_{e_m}\circ\cdots\circ\mathscr T_{e_1}$.
Corollary~\ref{cor:iterated} shows that the accumulated translation part is
exactly the sum of the transported rank one logarithmic anomalies.
In particular, whenever two marked flop paths connect the same chambers
and induce the same divisor-lattice identification, their accumulated
anomalies agree modulo Novikov-independent terms.  Thus the centered Gram
forcing is a flat section of this affine wall-crossing system.  The word
``flat'' refers only to consistency of the transition maps under chamber
relations; no nontrivial cohomology class is asserted.

\subsection{GV width distribution and moment reconstruction}

We now prove the numerical assertions of Theorem~\ref{thm:width-main} and
make explicit the distinction between the raw GV spectrum and the weighted
width distribution.  Suppose that the flop contracts a smooth irreducible
curve $C$ and write $n_d=n_{d[C]}$.  Toda \cite{TodaWidth} proves
\begin{equation}\label{eq:Toda-width}
\wid(C)
=\dim_{\mathbb C}A_{\rm con}
=\sum_{d\ge1}d^2n_d,
\qquad
\cwid(C)=n_1.
\end{equation}

The sequence $(n_d)$ is the genus zero GV spectrum.  We set
\begin{equation}\label{eq:width-distribution}
w_d(C)=d^2n_d
\end{equation}
and call $(w_d(C))$ the GV width distribution.  Its finite generating
polynomial is
\begin{equation}\label{eq:width-poly}
\mathsf W_C(z)
=\sum_{d\ge1}w_d(C)z^d
=\sum_{d\ge1}d^2n_dz^d.
\end{equation}
Let $\Theta=z\,d/dz$.  Then
\begin{equation}\label{eq:kappa-width}
[z]\mathsf W_C(z)=\cwid(C),
\qquad
\mathsf W_C(1)=\wid(C),
\qquad
(\Theta\mathsf W_C)(1)=\kappa_C.
\end{equation}
Thus there are two equivalent moment conventions worth keeping separate:
from the GV spectrum viewpoint, $\wid(C)$ and $\kappa_C$ are respectively
the second and third moments; from the width distribution viewpoint,
$\wid(C)$ is the total mass and $\kappa_C$ is the first moment.

When $\wid(C)>0$, normalize the width distribution by
\begin{equation}\label{eq:normalized-width}
p_d=\frac{w_d(C)}{\wid(C)}.
\end{equation}
Then $\sum_dp_d=1$, and
\begin{equation}\label{eq:width-mean}
\mathbb E_{\rm wid}[d]
=\sum_ddp_d
=\frac{\kappa_C}{\wid(C)}.
\end{equation}
If $\ell=\max\{d:n_d\ne0\}$, positivity immediately yields
\begin{equation}\label{eq:width-bounds}
1\le\frac{\kappa_C}{\wid(C)}\le\ell,
\qquad
\wid(C)\le\kappa_C\le\ell\,\wid(C).
\end{equation}
This gives a useful interpretation of the anomaly-to-width ratio: it is
the mean GV degree seen by the width distribution.  In particular,
$\kappa_C$ increasingly emphasizes higher degree thickening of the
exceptional curve.

The relation with flop covariance is now transparent.  A degree $d$ BPS
sector contributes the weight $d^2n_d$ after the two divisor directions in
the Gram observable, while analytic continuation across the flopped ray
contributes one further factor of $d$ through
$\log(r^d)=d\log r$.  Hence the wall-crossing coefficient is the first
moment of the width distribution,
\[
\sum_d d\,w_d(C)=\sum_dd^3n_d=\kappa_C,
\]
which is precisely the coefficient in Theorem~\ref{thm:flop-cov}.  This is
the enumerative reason that the rank one anomaly measures the third GV
moment rather than Toda's second moment.

\begin{proof}[Proof of Theorem~\ref{thm:width-main}\textup{(i)--(iii)}]
The identities in part \textup{(i)} follow from
\eqref{eq:Toda-width} and termwise application of $\Theta$ to
\eqref{eq:width-poly}.  Part \textup{(ii)} is
\eqref{eq:normalized-width}--\eqref{eq:width-bounds}.  For part
\textup{(iii)}, write
\[
c=\cwid(C)=n_1,\qquad w=\wid(C),\qquad\kappa=\kappa_C.
\]
If $n_d=0$ for $d\ge3$, then
\[
w=c+4n_2,\qquad \kappa=c+8n_2,
\]
so
\begin{equation}\label{eq:length2-moment}
n_2=\frac{w-c}{4},
\qquad
\kappa=2w-c.
\end{equation}
If $n_d=0$ for $d\ge4$, then
\[
w=c+4n_2+9n_3,
\qquad
\kappa=c+8n_2+27n_3.
\]
Subtracting twice the first equality from the second and adding $c$ gives
$9n_3=\kappa-2w+c$, and substitution gives
\[
n_3=\frac{\kappa-2w+c}{9},
\qquad
n_2=\frac{3w-\kappa-2c}{4}.
\]
This proves the reconstruction formulas and shows that degree three is the
first level at which $\kappa_C$ can carry independent numerical information.
\end{proof}

Two further identities make the same threshold visible without assuming a
support bound:
\begin{equation}\label{eq:moment-differences}
\kappa_C-\wid(C)
=\sum_{d\ge2}d^2(d-1)n_d,
\end{equation}
\begin{equation}\label{eq:moment-differences-2}
\kappa_C-2\wid(C)+\cwid(C)
=\sum_{d\ge3}d^2(d-2)n_d\ge0.
\end{equation}
The second expression vanishes exactly at the numerical level when the
width distribution has no contribution beyond degree two.

\subsection{Examples and sensitivity to higher degrees}

For the standard length two Laufer spectrum $(n_1,n_2)=(5,1)$
\cite{BrownWemyss,Collinucci},
\[
\mathsf W_{\rm L}(z)=5z+4z^2,
\qquad
(\cwid,\wid,\kappa)=(5,9,13).
\]
The normalized width distribution is
\[
(p_1,p_2)=\left(\frac59,\frac49\right),
\qquad
\mathbb E_{\rm wid}[d]=\frac{13}{9}.
\]
Here $\kappa=2\wid-\cwid$, so the anomaly contains no independent numerical
information beyond the two widths.

For the length three $E_6$ spectrum $(n_1,n_2,n_3)=(6,3,1)$
\cite{Collinucci},
\[
\mathsf W_{E_6}(z)=6z+12z^2+9z^3,
\qquad
(\cwid,\wid,\kappa)=(6,27,57).
\]
The reconstruction formulas give
\[
n_3=\frac{57-54+6}{9}=1,
\qquad
n_2=\frac{81-57-12}{4}=3.
\]
Thus the third GV moment becomes genuinely independent as soon as degree
three BPS sectors are allowed.

More generally, if the flopping curve has scheme-theoretic length $\ell$
and the positivity statement $n_d\ge1$ holds for $1\le d\le\ell$, then
\begin{equation}\label{eq:kappa-length-bound}
\kappa_C
\ge
\sum_{d=1}^{\ell}d^3
=
\left(\frac{\ell(\ell+1)}{2}\right)^2.
\end{equation}
Hence the anomaly coefficient is increasingly sensitive to higher degree
thickening of the exceptional curve.

\subsection{Cohomological width series}

The modern BPS package is closely related to stable pair and Hall algebra
curve counting \cite{PandharipandeThomasCurve,PandharipandeThomasStable,BridgelandHall}
and to vanishing cycle GV theory \cite{MaulikToda}.  Davison \cite{Davison}
defines cohomological BPS invariants
\[
\BPS_{A_{\rm con},d}
\]
for the contraction algebra and proves the numerical identity
\begin{equation}\label{eq:Davison-dim}
\dim \BPS_{A_{\rm con},d}=n_d.
\end{equation}
We stress that only this numerical statement is used below; a stronger
identification with a separately defined geometric cohomological GV object
is not required.

Let $K_0(\mathrm{MMHS})$ be the Grothendieck group of monodromic mixed
Hodge structures and define
\begin{equation}\label{eq:coh-width}
\mathbb W_C(z)
=
\sum_{d\ge1}d^2[\BPS_{A_{\rm con},d}]z^d
\in K_0(\mathrm{MMHS})[z].
\end{equation}
The dimension realization gives
\begin{equation}\label{eq:coh-width-dim}
\dim\mathbb W_C(z)=\mathsf W_C(z).
\end{equation}
Its first width moment is
\begin{equation}\label{eq:coh-Gram-moment}
\mathbb K_C
=
\left.\Theta\mathbb W_C(z)\right|_{z=1}
=
\sum_{d\ge1}d^3[\BPS_{A_{\rm con},d}].
\end{equation}
Applying the dimension realization termwise gives
\begin{equation}\label{eq:coh-kappa}
\dim\mathbb K_C=\kappa_C.
\end{equation}
This proves Theorem~\ref{thm:width-main}\textup{(iv)}.

Thus the numerical coefficient controlling the Gram wall-crossing anomaly
is the numerical realization of an intrinsic BPS moment of the contraction
algebra.  Hua--Toda \cite{HuaToda} show, in the smooth irreducible setting,
that the contraction algebra together with its natural $A_\infty$ structure
recovers the genus zero GV invariants.  Hence it determines the full
numerical width polynomial and all of its moments.

The discussion may be summarized by the following hierarchy:
\[
\boxed{
(n_d)
\ \longrightarrow\
(w_d=d^2n_d)
\ \longrightarrow\
\mathsf W_C(z)
\ \longrightarrow\
\kappa_C
\ \longrightarrow\
\mathbb K_C
}.
\]
The first arrow passes from the GV spectrum to the width distribution, the
middle arrows extract numerical width moments, and the last arrow records
the cohomological lift supplied by contraction algebra BPS theory.

\begin{remark}
It is natural to ask for a cohomological enhancement of the centered Gram
forcing whose wall-crossing coefficient is $\mathbb K_C$ itself.  We do
not construct such an enhancement here.
\end{remark}

\section{Cartan reconstruction, ADE foldings and applications}

\subsection{Signed Cartan reconstruction}

Let $\Phi$ be a reduced finite root system with simple roots
$\alpha_1,\ldots,\alpha_r$ and Cartan matrix
\[
A=(a_{ij}),
\qquad
a_{ij}=\langle\alpha_i,\alpha_j^\vee\rangle.
\]
Take the difference directions to be the simple coroots
$\eta_i=\alpha_i^\vee$.

Define the entrywise signed $q$-Cartan profile
\[
A_u=([a_{ji}]_u)_{j,i}.
\]
We use this phrase to avoid conflating $A_u$ with any one of the standard
notions of a deformed Cartan matrix.

For the simple root $\alpha_j$, put
\[
\mathbf c_j(u)
=
([a_{j1}]_u,\ldots,[a_{jr}]_u)^T.
\]

By Theorem~\ref{thm:quantum-Gram}, the primitive coefficient on the
simple-root ray is exactly the rank one matrix in
\eqref{eq:qCartan-row}.  Set
\[
M^{(j)}
=
-[q^{\alpha_j}]\mathbf H.
\]
Since $a_{jj}=2$,
\[
M^{(j)}_{ij}
=
N_{\alpha_j}[a_{ji}]_u[2]_u,
\qquad
M^{(j)}_{jj}
=
N_{\alpha_j}[2]_u^2.
\]

Dividing these two identities gives \eqref{eq:signed-ratio}; taking the
formal germ at $u=1$ gives \eqref{eq:signed-classical}.  This proves the
signed reconstruction assertion of Theorem~\ref{thm:Cartan-main}.
The statement is made for Laurent polynomial profiles, rather than an
arbitrary numerical specialization, because special values of $u$ can
produce accidental degeneracies.

\subsubsection{Global Gram factorization}

Let
\[
D_N=\diag(N_{\alpha_1},\ldots,N_{\alpha_r})
\]
and define the simple sector Gram matrix
\[
\mathcal G(u)
=
-\sum_{j=1}^r
[q^{\alpha_j}]\mathbf H.
\]

Summing the simple-root rank one matrices gives
\eqref{eq:global-Gram}, and hence
\[
\det\mathcal G(u)
=
\left(
\prod_jN_{\alpha_j}
\right)
\det(A_u)^2,
\qquad
\mathcal G(1)=A^TD_NA.
\]
This proves the global factorization assertion of
Theorem~\ref{thm:Cartan-main}.

The closest algebraic analogy is with deformed Cartan matrices in
representation theory; see, for example, \cite{FujitaMurakami}.  The
difference is that the Gromov-Witten sine-square denominator naturally
produces the quadratic Gram object \eqref{eq:global-Gram}.

\subsection{ADE foldings and non-simply-laced applications}

For the standard folded ADE potential,
\[
N_\beta=d_\beta=\frac2{(\beta,\beta)}.
\]
For a simple folded root,
\[
M^{(j)}
=
d_j\,\mathbf c_j(u)\mathbf c_j(u)^T.
\]
Hence
\begin{equation}\label{eq:fold-reconstruct}
d_j
=
\frac{M^{(j)}_{jj}}{[2]_u^2},
\qquad
[a_{ji}]_u
=
[2]_u
\frac{M^{(j)}_{ij}}{M^{(j)}_{jj}}.
\end{equation}

Thus the same primitive coefficient matrix recovers both the root length
and the signed Cartan row.

\subsubsection{\texorpdfstring{$B_n$ versus $C_n$}{B n versus C n}}

For
\[
A_{2n-1}/\mathbb Z_2\to C_n,
\]
the final double bond satisfies
\[
(d_{n-1},d_n)=(2,1),
\qquad
(a_{n-1,n},a_{n,n-1})=(-1,-2).
\]
For
\[
D_{n+1}/\mathbb Z_2\to B_n,
\]
the data are
\[
(d_{n-1},d_n)=(1,2),
\qquad
(a_{n-1,n},a_{n,n-1})=(-2,-1).
\]
The centered Gram forcing therefore distinguishes the two orientations by
the interchange
\[
[-1]_u\longleftrightarrow[-2]_u.
\]

\subsubsection{\texorpdfstring{$F_4$ and $G_2$}{F4 and G2}}

Use
\[
A_{F_4}
=
\begin{pmatrix}
2&-1&0&0\\
-1&2&-1&0\\
0&-2&2&-1\\
0&0&-1&2
\end{pmatrix},
\qquad
D_{F_4}=\diag(2,2,1,1).
\]
The double bond is detected by $[-1]_u=-1$ in one direction and
\[
[-2]_u=-(u+u^{-1})
\]
in the other.

For
\[
D_4/\mathbb Z_3\to G_2,
\]
take
\[
A_{G_2}
=
\begin{pmatrix}
2&-1\\
-3&2
\end{pmatrix},
\qquad
D_{G_2}=\diag(3,1).
\]
The triple incidence is
\[
[-3]_u=-(u^2+1+u^{-2}).
\]
Thus $D_4$ triality is detected independently by the short root orbit
weight $3$ and by the signed Cartan entry $-3$.

\subsubsection{\texorpdfstring{Global $q$-Gram discriminants}{Global q-Gram discriminants}}

Set
\[
s=[2]_u=u+u^{-1}.
\]

For $B_n$ let
\[
A_u(B_n)
=
\begin{pmatrix}
s&-1&&&&\\
-1&s&-1&&&\\
&\ddots&\ddots&\ddots&&\\
&&-1&s&-s\\
&&&-1&s
\end{pmatrix},
\]
and $A_u(C_n)=A_u(B_n)^T$.  Put
\[
\delta_n(u)
=
\det A_u(B_n)
=
\det A_u(C_n).
\]

\begin{proof}[Proof of Theorem~\ref{thm:BC-det}]
Let $P_k$ be the determinant of the leading simply laced tridiagonal
$k\times k$ block.  Then
\[
P_0=1,\qquad P_1=s,\qquad
P_k=sP_{k-1}-P_{k-2},
\]
hence $P_k=[k+1]_u$.  The final double bond contributes the product $s$ of
its off diagonal entries, so
\[
\delta_n=sP_{n-1}-sP_{n-2}
=[2]_u([n]_u-[n-1]_u).
\]
The Laurent form follows by elementary simplification.
\end{proof}

Since
\[
\det D_{B_n}=2,
\qquad
\det D_{C_n}=2^{n-1},
\]
Theorem~\ref{thm:global-Gram} gives
\[
\det\mathcal G_{B_n}(u)
=
2\,\delta_n(u)^2,
\qquad
\det\mathcal G_{C_n}(u)
=
2^{n-1}\delta_n(u)^2.
\]
Thus
\[
\frac{\det\mathcal G_{C_n}(u)}
{\det\mathcal G_{B_n}(u)}
=
2^{n-2}.
\]

For $F_4$,
\[
A_u(F_4)
=
\begin{pmatrix}
s&-1&0&0\\
-1&s&-1&0\\
0&-s&s&-1\\
0&0&-1&s
\end{pmatrix},
\]
and direct computation gives
\[
\det A_u(F_4)
=
s^4-s^3-2s^2+1.
\]
Hence
\[
\det\mathcal G_{F_4}(u)
=
4(s^4-s^3-2s^2+1)^2.
\]

For $G_2$,
\[
A_u(G_2)
=
\begin{pmatrix}
[2]_u&-1\\
-[3]_u&[2]_u
\end{pmatrix}.
\]
Since $[3]_u=[2]_u^2-1$,
\[
\det A_u(G_2)=1
\]
and therefore
\begin{equation}\label{eq:G2-rigidity}
\det\mathcal G_{G_2}(u)=3.
\end{equation}
Thus the $G_2$ simple sector Gram discriminant is rigid under the
$q$-deformation, although its individual entries are not.

\subsection{Discussion and further directions}

The construction developed here has four complementary levels.

First, the half step centered differences lift scalar Fej\'er cancellation
to a matrix-valued $q$-Gram decomposition.  Primitive sectors are detected
by intrinsic rank one conditions, and tensor-valued divisor inversion
recovers the BPS data on each ray.

Second, the Gram observable interacts nontrivially with threefold
birational geometry.  Under a simple flop it is covariant away from the
flopped ray, while the ray contributes a universal logarithmic rank one
anomaly.  This anomaly is a finite $\lambda$ primitive of the cubic
crepant transformation correction.

Third, the width-distribution viewpoint identifies the enumerative content
of that anomaly.  The weighted sequence $w_d=d^2n_d$ has total mass
$\wid(C)$ and first moment $\kappa_C$; equivalently these are the second and
third moments of the underlying GV spectrum.  After normalization,
$\kappa_C/\wid(C)$ is the mean GV degree seen by the width distribution.
This explains both the inequalities in Theorem~\ref{thm:width-main} and
the enhanced sensitivity of the flop anomaly to higher degree BPS sectors.
Davison's cohomological BPS invariants lift the numerical width polynomial
to the intrinsic series $\mathbb W_C(z)$ and the anomaly coefficient to
$\mathbb K_C$.

Fourth, for root-supported theories the primitive Gram matrices recover the
signed Cartan data.  Standard ADE foldings then produce the
non-simply-laced $B$, $C$, $F$ and $G$ structures together with their global
$q$-Gram discriminants.  Thus the same matrix-valued difference observable
simultaneously records enumerative, birational, noncommutative and root-theoretic information.

Several questions remain.  It would be natural to construct a
cohomological refinement of the centered Gram forcing whose simple flop
correction is the cohomological moment
\[
\mathbb K_C
=
\sum_{d\ge1}d^3[\BPS_{A_{\rm con},d}].
\]
It would also be interesting to understand whether the global
$q$-Gram determinants admit a direct interpretation in quantum
representation theory or in the spectral geometry of the movable cone
arrangement.

\section*{Acknowledgements}

The author is supported by the National Natural Science Foundation of China
(NSFC) under Grant Nos.~11501470, 11426187, and 11791240561, and is partially
supported by NSFC Grant No.~11671328, the Chengdu Science and Technology
Program under Grant No.~2025-YF09-00007-SN, and the Fundamental Research
Funds for the Central Universities under Grant Nos.~2682021ZTPY043,
2682025ZTPY001, 2682025ZTPY057, and 2682025ZTO002. The author would like to thank Sung-Soo Kim, Kimyeong Lee, Futoshi Yagi,
Satoshi Nawata and Rui-Dong Zhu for useful discussions. Special thanks are due to
Bohui Chen, An-Min Li, and Guosong Zhao for their constant supports.
The author also gratefully acknowledges the many friends and colleagues
met at various conferences for their friendship, encouragement, and
stimulating conversations. Finally, the author would like to express special thanks to the Mainz
Institute for Theoretical Physics (MITP) of the Cluster of Excellence
PRISMA$^{+}$ (Project ID~390831469) for its hospitality and support.

\end{document}